\documentclass[draft]{amsart}

\usepackage{amssymb}
\usepackage[all]{xy}

\newtheorem{thm}{Theorem}[section]

\newtheorem{prop}[thm]{Proposition}
\newtheorem{lem}[thm]{Lemma}

\theoremstyle{definition}
\newtheorem{dfn}[thm]{Definition}
\newtheorem{ex}[thm]{Example}

\theoremstyle{remark}
\newtheorem{rem}[thm]{Remark}

\numberwithin{equation}{section}

\begin{document}

\title{$F$-signature of graded Gorenstein rings}

\author{Akiyoshi Sannai}
\author{Kei-ichi Watanabe}
\address{Graduate School of Mathematical Sciences, The University of Tokyo 
3-8-1 Komaba, Meguro, Tokyo, 153-8914 Japan}

\email[Akiyoshi Sannai]{sannai@ms.u-tokyo.ac.jp}

\address{Department of Mathematics,College of Humanities and Sciences,Nihon University,
Setagaya-ku,Tokyo, 156-0045 Japan}

\email[Kei-ichi Watanabe]{watanabe@math.chs.nihon-u.ac.jp}

\begin{abstract}
For a commutative ring $R$, the $F$-signature was defined 
 by Huneke and Leuschke \cite{H-L}. It is an invariant that measures the order of the rank of the free direct summand of $R^{(e)}$. Here, $R^{(e)}$ is $R$ itself, regarded as an $R$-module through $e$-times Frobenius action $F^e$.

In this paper, we show a connection 
of the F-signature of a graded ring 
with other invariants. More precisely, for a graded $F$-finite Gorenstein ring $R$ of dimension $d$, we give an inequality among the $F$-signature $s(R)$, $a$-invariant $a(R)$ and Poincar\'{e} polynomial $P(R,t)$.
\[ s(R)\le\frac{(-a(R))^d}{2^{d-1}d!}\lim_{t\rightarrow 1}(1-t)^dP(R,t) \]

Moreover, we show that $R^{(e)}$ has only one free direct summand for any $e$, if and only if $R$ is $F$-pure and $a(R)=0$. This gives a characterization of such rings.
\end{abstract}

\maketitle

\section{Introduction}

Let $R$ be a reduced local ring of dimension $d$ containing a field of characteristic $p>0$ with a perfect residue field. Let $R^{(e)}$ be $R$ itself, regarded as an $R$-module through the $e$-times composition of Frobenius maps $F^e$. In the following, we often assume that $R^{(e)}$ is $F$-finite. This is equivalent to $R^{(1)}$ being a finite $R$-module. The structure of $R^{(e)}$ has a close relationship with the singularity, multiplicity and Hilbert-Kunz multiplicity of $R$. For example, if $R$ is a regular local ring of dimension $d$, then by 
a theorem of Kunz, $R^{(e)}$ becomes a free module of rank $p^{ed}$. The converse also holds, and thus, this property characterizes regular local rings.

In \cite{H-L},Huneke and Leuschke investigated how many free direct summands are contained in $R^{(e)}$ . Namely, they decomposed $R^{(e)}$ as $R^{(e)}=R^{a_q}\oplus M_q$ and examined $a_q$, where $M_q$ is an $R$-module that does not contain any free direct summand. 
The number $a_q$ is known to be independent of the decomposition. 
In this article, first we consider the case where $a_q=1$ for any $q=p^e$. 

We obtained the following result.
\begin{thm}
For any reduced graded $F$-finite Gorenstein ring $R$ such that $R$ has an isolated singularity and $R_0$ is a perfect field, the following are equivalent.
\begin{enumerate}
\item $R^{(e)}$ has only one free direct summand, for each positive integer $e$.
\item $R$ is $F$-pure and $a(R)=0$.
\end{enumerate}
\end{thm}

\begin{dfn}
Decompose $R^{(e)}$ as $R^{(e)}=R^{a_q}\oplus M_q$, where $M_q$ is an $R$-module that does not contain $R$ as a direct summand. If the limit exists, we set
\[ s(R)=\lim_{q\rightarrow \infty}\frac{a_q}{q^d}, \]
and call it $F$-{\it signature}.
\end{dfn}

This invariant was introduced by Huneke and Leuschke, and has been investigated by several researchers. In their paper \cite{H-L}, the following theorems were shown.

\begin{thm}\label{Thm1.3}
Let $R$ be an $F$-finite reduced Cohen-Macaulay local ring with a perfect residue field and whose $F$-signature $s(R)$ exists. Then the following are equivalent.
\begin{enumerate}
\item $s(R)=1$.
\item $R$ is a regular local ring.
\end{enumerate}
\end{thm}

\begin{thm}
Let $R$ be an $F$-finite reduced complete Gorenstein local ring with a perfect residue field. Then the following are equivalent.
\begin{enumerate}
\item $s(R)>0$.
\item $R$ is $F$-rational.
\end{enumerate}
\end{thm}

From these results, we can say that the $F$-signature contains enough information to characterize some kind of classes of commutative rings. Theorem \ref{Thm1.3} was generalized by Y. Yao to that $s(R)\ge1-\frac{1}{d!p^d}$ implies $R$ is regular (see \cite{Y}). Besides, as in the assumption of Theorem \ref{Thm1.3}, it is also a nontrivial question whether the $F$-signature exists or not. Several papers have dealt with this problem; for example, Gorenstein local rings are known to admit $F$-signatures.

In this article, we obtained the following result.
\begin{thm}\label{Thm5.4}
For any reduced graded $F$-finite Gorenstein ring $R$ and if $R_0$ is assumed to be a perfect field, we have the following inequality.
\begin{equation}
s(R)\le\frac{(-a)^de^{\prime}}{2^{d-1}d!}.
\label{EqThm5.4}
\end{equation}
Here, $d$ is the dimension of $R$, 
$a=a(R)$ is the $a$-invariant of $R$, 
and $e^{\prime}=\displaystyle{\lim}_{t\rightarrow 1}(1-t^d)P(R,t)$, 
where $P(R,t)$ is the Poincar\'{e} series of $R$.

\end{thm}

In dimension 2, the equality holds in $(\ref{EqThm5.4})$ if $R$ is regular or a rational double point. 
This equality never holds for regular 
rings of dimension greater than or equal to 3.

\section{Preliminaries}
In the following, we write $q=p^e$ for any positive integer $e$.

\begin{prop}
Let $R$ be a $d$-dimensional Cohen-Macaulay local ring of characteristic
 $p$. The following are equivalent.
\begin{enumerate}
\item $R$ is $F$-rational.
\item $0$ is tightly closed in $H_{\mathfrak{m}}^d(R)$.
\end{enumerate}
\end{prop}
\begin{proof}
See 
\cite{Sm}, \cite[Theorem 4.4]{H}.
\end{proof}

\begin{prop}
For any $F$-pure Cohen-Macaulay positively graded ring $R$ of dimension $d$ and characteristic $p$, we have $a(R)\le 0$.
\end{prop}
\begin{proof}
See \cite{F-W}, Remark 1.6.
\end{proof}

\begin{prop}
Let $R$ be an $F$-injective Cohen-Macaulay graded 
ring of dimension $d$ and characteristic $p$, which is an isolated singularity. If $a(R)<0$, then $R$ is $F$-rational.
\end{prop}
\begin{proof} 
Let $\tau(R)$ be the test ideal of $R$. Because $R$ has an isolated singularity,
$\tau(R)$ is $\mathfrak{m}=\oplus_{n>0} R_n$-primary. Hence the annihilator of $\tau(R)$ in 
$H_{\mathfrak{m}}^d(R)$ is finitely generated. 

If $R$ is not $F$-rational, then there exists a non-zero element $x$ in 
$H_{\mathfrak{m}}^d(R)$ that satisfies $cx^q=0$ for all $c\in\tau(R)$ and 
$q$. Since $a(R)<0$ and $R$ is F-injective, $x^q\ne 0$ and 
$\deg x^q$ gets arbitrary small. Hence $\tau(R) x^q$ cannot be $0$. A contradiction! 
\end{proof}

\section{$F$-signature}
Let $R$ be a reduced commutative ring of positive characteristic with a perfect residue field. As before, let $p$ be $\mathrm{char} R$ and let $q=p^e$ for any positive integer $e$.

In the following, we assume $R$ is $F$-finite, i.e., $R^{(1)}$ is finite as an $R$-module.

\begin{dfn}
Decompose $R^{(e)}$ as $R^{(e)}=R^{a_q}\oplus M_q$, where $M_q$ is an $R$-module which does not contain $R$ as a direct summand. If the limit exists, we set
\[ s(R)=\lim_{q\rightarrow \infty}\frac{a_q}{q^d}, \]
and call it $F$-{\it signature}.
\end{dfn}

\begin{rem}
For each $e$, we have $a_q\le\mathrm{rank}R^{(e)}\le q^d$. Thus $s(R)$ satisfies $s(R)\le 1$.
\end{rem}

\begin{ex}\label{Ex4.3}
Let $R$ be a regular local ring of dimension $d$ with a perfect residue field. By the 
theorem of Kunz \cite{Ku}, 
we have $R^{(e)}=R^{q^d}$. Thus we have $a_q=q^d$, and $s(R)=1$.

The converse also holds. See \cite[Corollary 16]{H-L}.
\end{ex}

\begin{prop}\label{Prop4.4}
Let $\hat{R}$ be the $\mathfrak{m}$-adic completion of $R$. Then we have $s(R)=s(\hat{R})$.
\end{prop}
\begin{proof}
This is trivial, because the completion is fully faithful.
\end{proof}

\begin{thm}
Let $R$ be a Gorenstein ring. The following are equivalent.
\begin{enumerate}
\item $s(R)>0$.
\item $R$ is $F$-rational.
\item $R$ is $F$-regular.
\end{enumerate}
\end{thm}
\begin{proof}
See \cite[Theorem 11]{H-L}.
\end{proof}

\begin{ex}
For a 2-dimensional rational double point, we can calculate its $F$-signature several ways. See, for example, \cite{W-Y}. The results are listed below.
\begin{center}
\begin{tabular}
[c]{cccc}
type    &  equation            & $\mathrm{char}R$ & $s(R)$             \\\hline
$(A_n)$ & $f=x^2+y^2+z^{n+1}$  & $p\ge2$          & $\frac{1}{n+1}$    \\
$(D_n)$ & $f=x^2+yz^2+y^{n-1}$ & $p\ge3$          & $\frac{1}{4(n-2)}$ \\
$(E_6)$ & $f=x^2+y^3+z^4$      & $p\ge5$          & $\frac{1}{24}$     \\
$(E_7)$ & $f=x^2+y^3+yz^3$     & $p\ge5$          & $\frac{1}{48}$     \\
$(E_8)$ & $f=x^2+y^3+z^5$      & $p\ge7$          & $\frac{1}{120}$
\end{tabular}
\end{center}
\end{ex}

\section{Proof of the theorem}
In this section, we give a proof of the main theorem and calculate some examples. We use the notation of the previous section. 
Throughout this section, let $R=\oplus_{n\ge 0} R_n$ be an F-finite graded 
ring of positive characteristic. We denote $\mathfrak{m}=R_+=
\oplus_{n>0} R_n$. 
  
\begin{lem}\label{Lem5.1}
Let $R$ be a graded Gorenstein ring, and let $a=a(R)$. Then, for any homogeneous element $\alpha$ in $R$ such that $\alpha F^e$ is split mono, we have
\[ \deg\alpha\le-a(q-1). \]
\end{lem}
\begin{proof}
Let $f_1,f_2,\ldots, f_d$ be a homogeneous system of parameter $R$, and let $b_i=\deg f_i$. Because $\alpha F^e$ is split mono, by tensoring $R^q/(f_1^q,f_2^q,\ldots, f_d^q)$ over $R^q$, we see that $\alpha R^q/(f_1^q,f_2^q,\ldots, f_d^q)$ is a free direct summand of 
$R/(f_1^q,f_2^q,\ldots, f_d^q)$.

If we take a generator $z$ of $\mathrm{soc}(R/(f_1,f_2,\ldots, f_d))$, then it satisfies
\[ \deg z=a+b_1+b_2+\cdots+b_d . \]
Because $\alpha R^q/(f_1^q,f_2^q,\ldots, f_d^q)$ is a free direct summand of $R/(f_1^q,f_2^q,\ldots, f_d^q)$, we have $\alpha z^q\ne 0$
in $R/(f_1^q,f_2^q,\ldots, f_d^q)$, and thus
\[ \deg\alpha z^q\le a(R/(f_1^q,f_2^q,\ldots, f_d^q)). \]
Because $a(R/(f_1^q,f_2^q,\ldots, f_d^q))=a+q(b_1+b_2+\cdots+b_d)$, we obtain
\[ \deg\alpha+q(a+b_1+b_2+\cdots+b_d)\le a+q(b_1+b_2+\cdots+b_d). \]
Thus Lemma \ref{Lem5.1} follows.
\end{proof}

\begin{thm}\label{Thm5.3}
Let $R$ be a reduced graded $F$-finite Gorenstein ring with an isolated singularity, and let us assume that $R_0=k$ is a perfect field.
The following are equivalent.
\begin{enumerate}
\item $R^{(e)}$ has only one free direct summand for each positive integer $e$.
\item $R$ is $F$-pure and $a(R)=0$.
\end{enumerate}
\end{thm}
\begin{proof}
Because the implication from {\rm (2)} to {\rm (1)} is obvious by Lemma \ref{Lem5.1}, it suffices to show the converse.

Suppose we have $a_q=1$ for some 
$e\gg 1$. This means there exists an $\alpha$ such that $\alpha R$
 is a free direct summand of $R^{(e)}$. $\alpha$ is not contained in $R^0$. Therefore, $R$ is $F$-pure(see \cite[page 128,remark(c)]{HH}), which implies that $a(R)\le 0$. If we suppose that $a(R)<0$, then $R$ becomes $F$-rational and $s(R)>0$, contradicting $a_q=1$.
\end{proof}

\begin{thm}\label{Thm5.4}
For any reduced graded $F$-finite Gorenstein ring $R$ and if we assume that $R_0=k$ is a perfect field, we have the following inequality.
\begin{equation}
s(R)\le\frac{(-a)^de^{\prime}}{2^{d-1}d!}.
\label{EqThm5.4}
\end{equation}
Here, $d$ is the dimension of $R$, $a$ is the $a$-invariant, and $e^{\prime}=\displaystyle{\lim}_{t\rightarrow 1}(1-t)^dP(R,t)$, where $P(R,t)$ is the Poincar\'{e} series of $R$.
\end{thm}
\begin{proof}
Let $\alpha_1,\alpha_2,\ldots,\alpha_{a_q}$ be a basis of the free 
direct summand of $R^{(e)}$. 
We assume that each $\alpha_i$ is a homogeneous element. 
Let $f_1,f_2,\ldots,f_d$ be homogeneous parameter systems of $R$, and let $b_i=\deg f_i$. 
We may assume that each $b_i$ is large enough with respect to $a_q$ and $q$.
Take a generator $z$ of $\mathrm{soc}(R/(f_1,f_2,\ldots, f_d))$.
By Lemma \ref{Lem5.1}, we have $\deg\alpha_i\le-a(q-1)$.
Because $\alpha_1R^q\oplus\alpha_2R^q\oplus\cdots\oplus\alpha_{a_q}R^q$ is a free direct summand of $R^{(e)}$,
\[ \alpha_1R^q/(f_1,f_2\ldots,f_d)\oplus\alpha_1R^q/(f_1,f_2\ldots,f_d)\oplus\cdots\oplus\alpha_{a_q}R^q/(f_1,f_2\ldots,f_d) \]
becomes a free direct summand of $R^{(e)}/(f_1,f_2\ldots,f_d)R^{(e)}=R/(f_1^q,\ldots, f_d^q)R$.

Define $a_q^-$ and $a_q^+$ by
\begin{eqnarray*}
a_q^-&=&\{ i\mid \deg\alpha_i<\frac{-a(q-1)}{2} \},\\
a_q^+&=&\{ i\mid \deg\alpha_i\ge\frac{-a(q-1)}{2} \}.
\end{eqnarray*}
And let $r_n$=dim$_k(R/(f_1^q,\ldots, f_d^q)R)_n$. Because $\alpha_i$'s are linearly independent over $k$, we have
\[ a_q^-\le\sum_{n=0}^{\frac{-a(q-1)}{2}-1}r_n. \]
Similarly, because $\alpha_iz^q$ 's are linearly independent over $k$, we have
\[ a_q^+\le\sum_{n=\frac{-a(q-1)}{2}+q(a+b_1+\cdots+b_d)}^{-a(q-1)+q(a+b_1+\cdots +b_d)}r_n. \]
By the duality of the Gorenstein ring $R^{(e)}/(f_1,\ldots, f_d)R^{(e)}=R/(f_1^q,\ldots, f_d^q)R$, we obtain
\[ r_{\frac{-a(q-1)}{2}+q(a+b_1+\cdots +b_d)+i}=r_{\frac{-a(q-1)}{2}-i}\]
and thus
\[ a_q^+\le\sum_{n=0}^{\frac{-a(q-1)}{2}}r_n. \]
Thus we obtain
\[ a_q=a_q^++a_q^-\le2\sum_{n=0}^{\frac{-a(q-1)}{2}}r_n. \]
Dividing his by $q^d$ and taking the limit, we obtain s(R) from the left-hand side.
For large $q$,
\[ 
2\sum_{n=0}^{\frac{-a(q-1)}{2}}r_n=2\frac{e^{\prime}}{d!}(\frac{-a(q-1)}{2})^d+(terms of lower degree)=\frac{(-a)^de^{\prime}}{2^{d-1}d!}q^d+(terms of lower degree). 
\]
Hence the result follows.
\end{proof}

\begin{ex}
If $R=k[X_1,\ldots, X_d]$ is a polynomial ring over perfect field $k$, because 
$a(R)=-d$ and $e^{\prime}=1$, we have
\[ \frac{(-a)^de^{\prime}}{2^{(d-1)d!}}=\frac{d^d}{2^{(d-1)}d!}. \]
The right hand side is equal to $1$ for $d=1,2$ and greater than $1$ for $d\ge 3$. Thus we have equality in {\rm (\ref{EqThm5.4})} if and only 
if $d=1,2$. 
 \end{ex}

\begin{ex}
$(\ref{EqThm5.4})$ becomes an equality if $R$ is a 2-dimensional rational double point.
\end{ex}
To confirm this, we introduce the following lemma.
\begin{lem}\label{Lem5.7}
Let $R$ be a graded ring, and let $x$ be a regular homogeneous element in degree $b$. Then
\begin{equation}
(1-t^b)P(R,t)=P(R/(x),t)
\label{EqLemma5.7}
\end{equation}
\end{lem}
\begin{proof}
Since $x$ is a homogeneous regular element of degree $b$, 
\begin{equation}
0\rightarrow R(-b)\overset{x}{\longrightarrow}R\rightarrow R/(x)\rightarrow 0
\label{SeqLem5.7}
\end{equation}
is exact. By the additivity of the dimension of $k$-vector spaces, we obtain
\[ H(R/(x),n)=H(R,n)-H(R,n-b), \]
and thus
\begin{eqnarray*}
H(R/(x),n)t^n&=&H(R,n)t^n-H(R,n-b)t^n,\\
H(R/(x),n)t^n&=&H(R,n)t^n-t^bH(R,n-b)t^{n-b}.
\end{eqnarray*}
Here, $H(R,n)$ denotes the $k$-dimension of the degree n part of $R$.
Taking the sum of these equations, we obtain $(\ref{EqLemma5.7})$.
\end{proof}

\begin{lem} \label{Lem5.8}
Let $R$ be a graded Cohen-Macaulay ring, and let $x$ be a homogeneous regular element of degree $b$. Then we have
\[ a(R/(x))=a(R)+b. \]
\end{lem}
\begin{proof} See \cite{G-W1}(2.2.10) or \cite{B-H}(3.6.14).
\end{proof}

In the following, we demonstrate how to calculate the right-hand side of $(\ref{EqThm5.4})$ in Theorem \ref{Thm5.4} for singularities of type $A_n$ by using these lemmas. In this case, $R=k[x,y,z]/x^2+y^2+z^{n+1}$. We remark that because $R$ is a hypersurface singularity, it is a complete intersection, in particular Cohen-Macaulay.

$R$ can be regarded as a graded ring with $\deg x=\deg y=n+1, \deg z=2$. With this grading, $x^2+y^2+z^{n+1}$ is a homogeneous element of degree $2(n+1)$, and thus we have
\[ a(R)=-(n+1+n+1+2)+2(n+1)=-2 \]
by Lemma \ref{Lem5.8}.

Next we calculate the Poincar\'{e} series. By Lemma \ref{Lem5.7}, we have
\[ (1-t^{n+1})P(R,t)=P(R/(x),t). \]
Moreover, because $y$ is a regular element in $R/(x)$ of degree $n+1$, we have
\[ (1-t^{n+1})P(R/(y),t)=P(R/(x,y),t) \]
again by Lemma \ref{Lem5.7}.
Since $z$ is an element of degree $2$ in $R/(x,y)=k[x,y,z]/z^{n+1}$, we have
\[ 
P(R/(x,y),t)=1+t^2+\cdots+t^{2n}=\frac{1-t^{2(n+1)}}{1-t^2}. 
\] 
Combining with the above two equations, we obtain
\[ P(R,t)=\frac{1-t^{2(n+1)}}{(1-t^{n+1})(1-t^{n+1})(1-t^2)}, \]
and thus $e^{\prime}=\frac{1}{n+1}$. The right-hand side of {\rm (\ref{EqThm5.4})} can be calculated to be $\frac{1}{n+1}$, and thus is equal to $e^{\prime}$ in this case.

The following is the list of $a$ and $e^{\prime}$ calculated by the above method. 
\begin{center}
\begin{tabular}
[c]{cccc}
type    & $e^{\prime}$    & $a(R)$ & {\rm RHS} of {\rm (\ref{EqThm5.4})} \\\hline
$(A_n)$ & $\frac{1}{n+1}$ & $-2$   & $\frac{1}{n+1}$    \\
$(D_n)$ & $\frac{1}{n-1}$ & $-1$   & $\frac{1}{4(n-2)}$ \\
$(E_6)$ & $\frac{1}{6}$   & $-1$   & $\frac{1}{24}$     \\
$(E_7)$ & $\frac{1}{12}$  & $-1$   & $\frac{1}{48}$     \\
$(E_8)$ & $\frac{1}{30}$  & $-1$   & $\frac{1}{120}$
\end{tabular}
\end{center}

\section{Local case}
Lastly, we give a local version of theorem \ref{Thm5.3} .

\begin{thm}
Let $(R,\mathfrak{m})$ be a $F$-finite Gorenstein local ring that is $F$-rational on a punctured spectrum and assume that the residue field is perfect. The following are equivalent.
\begin{enumerate}
\item $R^{(e)}$ has only one free direct summand, for each positive integer $e$.
\item $R$ is $F$-pure and not $F$-rational
\end{enumerate}
\end{thm}
\begin{proof}
The implication from {\rm (1)} to {\rm (2)} is the same as the graded case.
To show the converse, we use the splitting prime. For the definition and the behavior, see \cite{A-E2}.
Let $\mathfrak{p}$ be the splitting prime. Because $R$ is $F$-pure, $\mathfrak{p}$ is not unit ideal.
Because $\mathfrak{p}$ contains the test ideal, $\mathfrak{p}$ equals to $\mathfrak{m}$.(See \cite[remark3.5.]{A-E2}.)\\
But then $a_q=1$ by the remark again, namely $R^{(e)}$ has only one free direct summand, for each positive integer $e$.

\end{proof}

\section{Acknowledgement}
We would like to thank Anurag K. Singh and the refree for their useful advice.


\begin{thebibliography}{BR}
\bibitem[A-E1]{A-E1} Aberbach, I.M.; Enescu, F.: \emph{When does the $F$-signature exists?}. Ann. Fac. Sci. Toulouse Math. (6) \textbf{1}5 (2006), no. 2, 195--201.
\bibitem[A-E2]{A-E2} Aberbach, I.M.; Enescu, F.: \emph{The structure of $F$-pure rings}. Math. Z. \textbf{250} (2005), no. 4, 791--806. 

\bibitem[A-L]{A-L} Aberbach, I.M.; Leuschke, G.J.: \emph{The $F$-signature and strong $F$-regularity}. Math. Res. Lett. \textbf{10} (2003), no. 1, 51--56.

\bibitem[B-H]{B-H} Bruns, W.; Herzog, J.: \emph{Cohen-Macaulay rings}. Cambridge Studies in Advanced Mathematics, \textbf{39}. Cambridge University Press, Cambridge, 1993. xii+403 pp.

\bibitem[F-W]{F-W} Fedder, R.; Watanabe, K.-i.: \emph{A characterization of F-regularity in terms of F-purity}, \
\lq\lq Commutative Algebra", (Proc. Microprogram, Berkeley, 1987), 
227-245 (1989) Springer. 
 

\bibitem[G-W1]{G-W1} Goto, S.; Watanabe, K.: \emph{On graded rings. I}. J. Math. Soc. Japan \textbf{30} (1978), no. 2, 179--213.

\bibitem[G-W2]{G-W2} Goto, S.; Watanabe, K.: \emph{On graded rings. II}. ($Z^{n}$-graded rings). Tokyo J. Math. \textbf{1} (1978), no. 2, 237--261.

\bibitem[HH]{HH} Hochster, Melvin; Huneke, Craig: \emph{Tight closure and strong $F$-regularity}. Mem. Soc. Math. France (N.S.) (1989),no.38, 119-133.

\bibitem[H]{H}Huneke, C.: \emph{Tight closure and its applications}. With an appendix by Melvin Hochster. CBMS Regional Conference Series in Mathematics, \textbf{88}. Published for the Conference Board of the Mathematical Sciences, Washington, DC; by the American Mathematical Society, Providence, RI, 1996. x+137 pp. 

\bibitem[H-L]{H-L} Huneke, C.; Leuschke, G.J.: \emph{Two theorems about maximal Cohen-Macaulay modules}. Math. Ann. \textbf{324} (2002), no. 2, 391--404.

\bibitem [Ku]{Ku} Kunz, E.: 
\emph{Characterizations of regular local rings of characteristic $p$}, 
Amer. J. Math. \textbf{41} (1969), , 772--784. 

\bibitem[Si]{Si} Singh, A.K. \emph{The $F$-signature of an affine semigroup ring}. J. Pure Appl. Algebra \textbf{196} (2005), no. 2-3, 313--321.

\bibitem[Sm]{Sm} Smith, E. K., {\em F-rational rings have rational singularities}, Amer. J. Math. {\bf 119} (1997), 159--180.

\bibitem[W-Y]{W-Y} Watanabe, K.;Yoshida,K.: \emph{Minimal relative Hilbert-Kunz multiplicity}. Illinois J. Math. \textbf{48} (2004), no. 1, 273--294.

\bibitem[Y]{Y} Yao,Y.: \emph{Observations on the $F$-signature of local rings of characteristic $p$}. J. Algebra \textbf{299} (2006), no. 1, 198--218.
\end{thebibliography}
\end{document}